\documentclass{amsart}
\usepackage{amsmath}
\usepackage{paralist}
\usepackage{graphics}
\usepackage{epsfig}
\usepackage{verbatim}

\newtheorem{theorem}{Theorem}[section]

\theoremstyle{definition}
\newtheorem{definition}[theorem]{Definition}
\newtheorem{remark}{Remark}

\newcommand{\ep}{\varepsilon}
\newcommand{\Om}{\Omega}
\newcommand{\HH}{H^1(\Om) \times L^2(\Om)}

\newcommand{\re}{\mathbb{R}}

\newcommand{\HO}{H^1(\Om)}
\newcommand{\LO}{L^2(\Om)}

\title[Existence of weak solutions to the Cauchy problem]
      {Existence of weak solutions to the Cauchy problem of a semilinear wave equation with supercritical interior source and damping}

\author[Lorena Bociu and Petronela Radu]{}

\subjclass{Primary: 35L15, 35L70; Secondary: 35L05}
 \keywords{wave equations, damping and source terms, weak solutions, energy identity}

 \email{lvb9b@virginia.edu}
 \email{pradu@math.unl.edu}

\begin{document}
\maketitle

\centerline{\scshape Lorena Bociu }
\medskip
{\footnotesize
 \centerline{ University of Nebraska-Lincoln}
   \centerline{Lincoln, NE 68588-0130, USA}
}
\medskip

\centerline{\scshape Petronela Radu }
\medskip
{\footnotesize
 \centerline{ University of Nebraska-Lincoln}
   \centerline{Lincoln, NE 68588-0130, USA}
}

\bigskip

 \begin{abstract}
In this paper we show existence of finite energy solutions for the
Cauchy problem associated with a semilinear wave equation with
interior damping and supercritical source terms. The main contribution consists in dealing with super-supercritical source terms (terms of the order of $|u|^p$ with $p\geq 5$ in $n=3$ dimensions), an open and highly recognized problem in the literature on nonlinear wave equations.
\end{abstract}

\section{Introduction}
Consider the Cauchy problem:
   \begin{equation}\tag{SW}
   \left\{\begin{array}{l}
    u_{tt}-\Delta u + f(u) +g(u_t) = 0 \text{   a.e. } (x,t)\in \re^n
    \times [0,\infty);\\
    (u,u_t)|_{{}_{t=0}} = (u_{{}_0},u_{{}_1}) ,  \text{   a.e. }  x\in \re^n.
   \end{array} \right .
  \end{equation}

We are interested in the existence of weak solutions to (SW) on the
finite energy space $H^1(\re^n) \times L^2(\re^n)$. We will work
with the following notation : $| \cdot|_{p,\Omega} $ denotes the
$L^p(\Omega)$ norm, while for the $L^2$ norm we simply use $|
\cdot|_{\Omega}$; when there is no danger of confusion we simplify
the notation $| \cdot|_{p,\Omega} $ to $| \cdot|_{p}$.

For the sake of exposition, we will focus on the most relevant case of
dimension $n=3$, but the analysis can be adapted to any other value
of $n$. In this case, we classify the interior source $f$ based
on the criticality of the Sobolev's embedding $H^1(\re^3) \to
L^6(\re^3)$ as follows: (i) \textbf{subcritical}: $1 \leq p < 3$ and
\textbf{critical}: $p = 3$. In these cases, $f$ is locally Lipschitz
from $H^1(\re^3) $ into $ L^2(\re^3)$; (ii) \textbf{supercritical}:
$ 3 < p < 5$. For this exponent $f$ is no longer locally Lipschitz, but the potential
well energy associated with $f$ is still well defined on the finite
energy space; (iii) \textbf{super-supercritical}: $ 5 \leq p < 6$.
The source is no longer within the framework of potential well
theory, due to the fact that the potential energy may not be defined
on the finite energy space.

\subsection{Assumptions} Throughout the paper we will impose the following conditions on the source and damping terms:

\noindent\textbf{(A$_{\textbf{g}}$)} $g$ is increasing and continuous with $g(0)=0$. In addition, the following
growth condition  at infinity holds: there exist positive constants
$ \displaystyle l_m, L_m $ such that for $ \displaystyle |s|
> 1$ we have $\displaystyle l_m
|s|^{m+1} \leq g(s)s \leq L_m |s|^{m+1} $ with   $ \displaystyle m
\geq 0 $.\\

\noindent \textbf{(A$_{\textbf{f}}$)} $f\in C^1(\re)$ and the following growth
condition is imposed on $f$:
\[
|f'(u)|\leq C|u|^{p-1} \text{ for } |s| > 1
\]
where $p\in [1,6)$ satisfies either (a)$\,1<p \leq 3,\, m \geq 0\,\,\,\,$ or (b)$\,
p+\displaystyle\frac{p}{m} < 6/{(1+2\ep)}$ for some $\varepsilon>0$, where $m>0$ is the growth exponent
from (A$_g$).

\begin{remark}
Note that the Assumption (A$_f$) allows for both types of
supercriticality. Also, (A$_f$) guarantees that $f$ is locally Lipschitz from $H^{1-
\ep}(\re^3) \to L^{\frac{m+1}{m}}(\re^3)$.
\end{remark}

\begin{definition}\label{ws}
     Let $\Omega_T := \Omega \times(0,T), \, T>0 $, where $\Omega \subset \re^3$
is an open
connected set with smooth boundary $\partial \Omega$. Let $f$ and $g$ be two real valued functions $f$ and $g$ which satisfy (A$_f$) and (A$_g$), and further
suppose that
   $ u_0 \in H_0^1(\Omega)\cap L^{p+1}(\Omega)$  and  $u_1 \in L^2(\Omega)
     \cap L^{m+1}(\Omega)$.

   A weak solution on $\Omega_T$ of the boundary value problem
   \begin{equation}\tag{SWB}
  \left\{\begin{array}{l}
   u_{tt}-\Delta u+f(u)+ g(u_t)=0 \text{  in } \Omega_T;\\
   (u,u_t)|_{t=0}=(u_0,u_1)  \text{  in } \Omega;\\
   u=0 \text{ on } \partial \Omega \times (0,T).
   \end{array} \right .
   \end{equation}
   is any function $u$ satisfying
   \[
        u\in C(0,T; H_0^1(\Omega)) \cap L^{p+1}(\Omega _T), \quad
        u_t \in L^2(\Omega_T) \cap L^{m+1}(\Omega _T),
   \]
   and
   \begin{multline*}
      \int_{\Omega _T} \Big ( u(x,s)\phi_{tt}(x,s)+\nabla u(x,s) \cdot
\nabla
      \phi(x,s) +f(x,s,u)\phi(x,s) \\ +g(x,s,u_t)\phi(x,s)\Big)
      \, dx ds =\int_{\Omega} \Big(u_1(x)\phi(x,0)-u_0(x)\phi_t(x,0)\Big)
dx
         \end{multline*}
   for every $\phi \in C_c^{\infty}(\Omega \times (-\infty, T))$.
   \end{definition}

\begin{remark} A weak solution for the Cauchy problem (SW) is defined by taking in the above definition $\Omega=\re^n$ with no boundary conditions.
\end{remark}

\subsection{Relationship to previous literature. Significance of results.}

Semilinear wave equations with interior damping-source interaction
have attracted a lot of attention in recent years. In the case of
subcritical source $f$, local existence and uniqueness of solutions
are standard and they follow from monotone operator theory \cite{B}.
In \cite{GT}, the authors considered the case of polynomial damping
and source, i.e. $g(u_t) = |u_t|^{m-1}u_t$ and $f(u) = |u|^{p-1}u$
and showed that if the damping is strong enough ($ m \geq p$), the
solutions live forever, while in the complementary region $ m < p$,
the solutions blow-up in  finite time. For supercritical interior
sources, \cite{BLR}, \cite{fer} and \cite{PR2} exhibited existence
of weak solutions for a bounded domain $\Omega$, under the
restriction $p < 6m / (m+1)$, while \cite{STV} obtained the same
results for $ \Om = \mathbb R^3$, and compactly supported initial
data, with $p < 6m / (m+1)$. In this case, it was shown additionally
by \cite{PR1} that if the interior damping is absent or linear, the
exponent $p$ may be supercritical, i.e. $p < 5$; also, in \cite{PR1} the initial
data may not be compactly supported. The case of super-supercritical
sources on a bounded domain was analyzed and resolved recently in
\cite{B_exist}, \cite{BL_uniq}, \cite{BL_blowup}. The authors
considered the wave equation with interior and boundary damping and
source interactions, and proved existence and uniqueness of weak
solutions. Moreover, they provided complete description of
parameters corresponding to global existence and blow-up in finite
time. We will provide more details on these results in the next
section.

Our paper provides existence of solutions to wave equations on
$\re^3$ for the case of super-supercritical sources. The method used
will also provide an alternative proof in the case of supercritical
(and below) interior sources. Thus our paper extends the known
existence results to the super-supercritical case (we include the
dark shaded regions $ 5 \leq p < 6$). We summarize our results and
improvements over previous literature with the following
illustration:
\begin{center}
\includegraphics[width=2.5in]{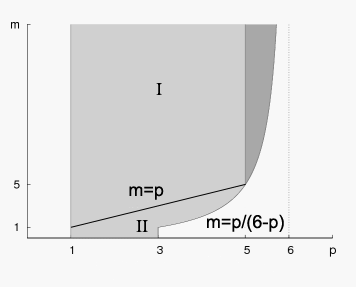}
\end{center}
Note that for the range of exponents $m\geq p$ (region I above) one expects global existence of solutions, while for $m < p$ the solutions may blow up in finite time (according to the preliminary results of \cite{GT, BLR,B_exist,BL_blowup} obtained on bounded domains).
\section{Preliminaries}
We include in this section the following theorems which were proved
in \cite{PR1} and which will be used in the proof of our main result.

  \begin{theorem}
{\bf(Existence and uniqueness of solutions for dissipative
    wave equations with Lipschitz source terms)}
\label{T:existLD1}
      Let $ \Omega \subset \re^3$ be a bounded domain with smooth boundary
 $\partial  \Omega$, and let the functions $f$ and $g$ satisfy
assumptions
{\rm  (A$_f$), (A$_g$),}
 where $f$ is globally Lipschitz. Let $u_0, u_1 \in
 H_0^1( \Omega) \times L^2( \Omega)$ and $T>0$ arbitrary. Then the boundary value problem
 \begin{equation}\tag{SWB}
   \left\{\begin{array}{l}
    u_{tt}-\Delta u + f(u) +g(u_t) = 0 \text{   a.e. } (x,t)\in  \Omega
    \times [0,\infty);\\
    u(x,t)=0,  \text{   a.e. } (x,t)\in \partial  \Omega  \times [0,\infty);\\
    (u,u_t)|_{{}_{t=0}} = (u_{{}_0},u_{{}_1}) ,  \text{   a.e. }  x\in  \Omega.
   \end{array} \right .
  \end{equation}
  admits a unique solution $u$ on the time interval $[0, T]$ in
the sense of the Definition $\ref{ws},$ i.e.,
 \[
        u\in C(0,T; H_0^1( \Omega)) \cap L^{p+1}( \Omega _T), \quad
        u_t \in L^2( \Omega_T) \cap L^{m+1}( \Omega _T).
   \]
\end{theorem}

The finite speed of propagation property is known to hold for wave equations
with nonlinear damping and/or with source terms {\it of good sign}, i.e.
their contribution to the energy of the system is decretive. The following
theorem states that the property remains true for source terms of {\it arbitrary} sign,
as long as they are Lipschitz (for a proof see \cite{PR1}).
\begin{theorem}
{\bf (Finite speed of propagation)}
\label{P:finitespeed}
    Consider the problem (SWB) under the hypothesis of Theorem
$\ref{T:existLD1}.$ Then

   $(1)$ if the initial data $u_0, u_1$ is compactly supported inside the
    ball $B(x_0, R)$ $ \subset  \Omega ,$ then $u(x,t)=0$ for all points $x\in \Omega$ outside B($x_0$, $R+t$);

   $(2)$ if $(u_0, u_1), (v_0, v_1)$ are two pairs of initial data with
    compact support, with the corresponding solutions $u(x,t)$,
    respectively $v(x,t)$, and $u_0(x)=v_0(x)$ for $x\in$ B($x_0$,$R)
\subset  \Omega$, then
    $u(x,t)=v(x,t)$ inside  B($x_0$, $R-t$) for any $t<R$.
   \end{theorem}

We conclude this section by stating the following result which appears
in  \cite{B_exist, BL_uniq} and whose analog on $\re^3$ we will prove in the next section.
\begin{theorem}\label{t:1}{\bf (Local existence and uniqueness in
the case of interior and boundary damping-source interactions)}
Consider the wave equation on an open bounded domain $\Om \subset
\mathbb{R}^3$
\begin{equation}\label{1.1}
\begin{cases}
                                  u_{tt} -\Delta u+ f(u)+g(u_t)=  0 $ in $  \Om
                                             \times [0, \infty)\\
                                 u =0$
                                          on $\partial\Omega  \times [0,\infty)\\
                                  u(0)=u_0 \in H_0^1(\Om)$ and $u_t(0)= u_1 \in
                                  L^2(\Om)
                       \end{cases}\end{equation}
under assumptions (A$_f$), (A$_g$) above. If $ \displaystyle p > 3 $, we
additionally assume that $\displaystyle f \in C^2(\mathbb{R})$, and
$\displaystyle |f''(s)| \leq C |s|^{p-2} $.
\noindent Then there exists a \textbf{local in time unique} weak
solution $u \in C[(0, T_M),H_0^1(\Omega)]\cap C^1[(0,T_M),L^2(\Omega)]$, where the maximal time of existence
$T_M>0 $depends on initial data $|(u_0, u_1)|_{H_0^1(\Om) \times L^2(\Om)}$,
and $l_m$ given by (A$_g$).
 \end{theorem}
 \begin{remark} The condition $\displaystyle |f''(s)| \leq C |s|^{p-2} $ is needed for the uniqueness, but not for the existence of solutions.
 \end{remark}

 \section{Local in time existence of solutions to the Cauchy problem}

Our main result states:
    \begin{theorem}\label{T:Exist} {\rm (Existence of weak solutions)}
   Let $(u_0, u_1) \in H_0^1(\re^3) \times L^2(\re^3)$ and consider the
Cauchy problem
    \begin{equation}\tag{SW}
   \left\{\begin{array}{l}
    u_{tt}-\Delta u + f(u) +g(u_t) = 0 \text{   a.e. in } \re^{n}
    \times [0,\infty);\\
    u|_{{}_{t=0}} = u_{{}_0};\ \ \ \
    u_{t}|_{{}_{t=0}} = u_{{}_1}.
   \end{array} \right .
  \end{equation}
   where $f$ and $g$ satisfy (A$_f$)-(A$_g$). Then, there exists a time $T>0$ such that (SW) admits a weak solution
on $[0,T]$ in the sense of Definition \ref{ws}. The existence time $T$ depends on the energy norm of the initial data and on the constant $l_m$ given by (A$_g$).
   \end{theorem}

   \begin{proof}[\bf Proof.]

   We identify the following steps:

\subsection{Local existence on bounded domains}
Consider for now the problem (SWB) where $\Omega$ is an open, bounded domain with smooth boundary. First we will solve the existence problem on such a domain; in the second step we will cut the initial data in pieces defined on small domains; finally, we will show how to piece together the solutions defined on these small domains to obtain existence of solutions on the entire space $\re^3$.

\textbf{Approximation of f}:  We consider the following approximation of equation (SWB), with $n \to \infty$ as the parameter of approximation:
\begin{equation}\label{1.2}
\begin{cases}
                                  u^n_{tt} - \Delta u^n + f_n(u^n)+ g(u^n_t)=0$ in $  \Om
                                             \times [0, \infty)\\
                                 u^n = 0 $
                                          on $\partial \Omega  \times [0,\infty)\\
                                  u^n(0)=u_0 \in \HO$ and $u^n_t(0)= u_1 \in \LO.
                       \end{cases}\end{equation}
We construct the approximating functions $f_n$ as follows: let
$\eta$ be a cutoff smooth function such that: $(i) 0 \leq \eta \leq
1$, $(ii) \eta(u) = 1$, if $|u| \leq n$, $(iii) \eta(u) = 0$, if $|u| >
2n$, and $(iv) |\eta'(u)| \leq C/n$. Now construct $f_n:
H^{1-\ep}(\Om) \to L^{\frac{m+1}{m}}(\Om)\ \text{,} \ f_n(u) :=
f(u)\eta(u)$. This means that
\begin{equation}\nonumber f_n(u) =
\begin{cases}
f(u) & \text{,}\ \  |u| \leq n\\
f(u) \eta(u) & \text{,} \ \ n < |u| < 2n.\\
0 & \text{, otherwise}.
\end{cases}
\end{equation}

\textbf{Claim 1}:  $f_n$ is locally Lipschitz from $H^{1-\ep}(\Om) \to
L^{\frac{m+1}{m}}(\Om)$ (uniformly in $n$). In the sequel we will use the notation $\tilde m =\frac{m+1}{m}$. In order to prove the claim, we consider the
following three cases:\\

\textbf{Case 1}: $|u|$ , $|v| \leq n$. Then $|f_n(u) -
f_n(v)|_{\tilde m} = |f(u) - f(v)|_{\tilde m}$ and we already
know that $f$ is locally Lipschitz from $H^{1-\ep}(\Om) \to L^{\tilde
m}(\Om)$.\\

\textbf{Case 2}: $ n \leq |u| \ \text{,} \ |v| \leq 2n$. Then we
have the following computations:
\begin{align}\label{casetwoapprox}
& |f_n(u) - f_n(v)|_{\tilde m} = |f(u)\eta(u) -
f(v)\eta(v)|_{\tilde m} \\ \nonumber & \leq |f(u)\eta(u) -
f(v)\eta(u) + f(v) \eta(u) - f(v)\eta(v)|_{\tilde m} \\ \nonumber &
\leq |f(u) - f(v)|_{\tilde m} + \Big(\int_{\Om} [|f(v)|\eta(u) -
\eta(v)|]^{\tilde m} \ dx \Big)^{m/(m+1)}\\ \nonumber & \leq
|f(u) - f(v)|_{\tilde m} + \Big(\int_{\Om} [|v|^{p-1}|v|
\max_{\xi}|\eta'(\xi)|u-v|]^{\tilde m} \ dx \Big)^{m/(m+1)}
\end{align}

Now using the definition of the cutoff function $\eta$ and the fact
that $|v| \leq 2n$, we can see that $|v| \displaystyle\max_{\xi}|\eta'(\xi) \leq C$ and
thus (\ref{casetwoapprox}) becomes
\begin{align}\label{approximationcasetwo}
& |f_n(u) - f_n(v)|_{\tilde m} \leq |f(u) - f(v)|_{\tilde m} +
\Big(\int_{\Om} |v|^{(p-1)\tilde m}|u-v|^{\tilde m} \ dx
\Big)^{m/(m+1)}.
\end{align}

For the second term on the right side of (\ref{approximationcasetwo}),
we use H\"older's Inequality with $p$ and $p/(p-1)$, the fact that
$p(m+1)/m \leq 6/(1+2\ep)$, and Sobolev's Imbedding $H^{1-\ep}(\Om)
\to L^{\frac{6}{1+2\ep}}(\Om)$ to obtain
\begin{align}
& |f_n(u) - f_n(v)|_{\tilde m} \leq |f(u) - f(v)|_{\tilde m} + C
|v|^{p-1}_{\frac{6}{1+2\ep}}
|u-v|_{\frac{6}{1+2\ep}}\\ \nonumber & \leq |f(u) -
f(v)|_{\tilde m} + C |v|^{p-1}_{H^{1-\ep}(\Om)}
|u-v|_{H^{1-\ep}(\Om)}
\end{align}
which proves that $f_n$ is locally Lipschitz $H^{1-\ep}(\Om) \to
L^{\frac{m+1}{m}}(\Om)$.\\

\textbf{Case 3}: If $|u| \leq n$ and $ n < v \leq 2n$, then we have
\begin{equation}\label{approxcasethree}
\begin{split}
|f_n(u) - &  f_n(v)|_{\tilde m} = |f(u) -
f(v)\eta(v)|_{\tilde m}\\
& \leq |f(u) - f(v)|_{\tilde m} + \Big( \int_{\Om}
|f(v)|1-\eta(v)| \ dx \Big)^{m/(m+1)}.
\end{split}
\end{equation}
In (\ref{approxcasethree}), we can replace $1 = \eta(u)$, since $|u|
\leq n$ and then the calculations follow exactly as in case 2.\\

\textbf{Claim 2:} For each $n$, $f_n$ is Lipschitz from $\HO \to
\LO$. Again, we consider the three cases:\\

\textbf{Case 1}: $|u| \leq n$ and $|v| \leq n$. Then
\begin{align}\label{approxinlo}
 |f_n(u) - f_n(v)|_{\Om} &= |f(u) - f(v)|_{\Om}\\
                                      & \leq \Big(
\int_{\Om} C |u-v|^2 [|u|^{p-1} + |v|^{p-1} + 1]^2 \ dx
\Big)^{1/2}.
\end{align}
Using H\"older's Inequality with 3 and 3/2, the fact that $|u| \leq n$
and $|v| \leq n$ and Sobolev's Imbedding $\HO \to L^6(\Om)$,
(\ref{approxinlo}) becomes
\begin{equation}
|f_n(u) - f_n(v)|_{\Om} \leq C_n |\nabla(u-v)|_{\Om}
\end{equation}

\textbf{Case 2}: $ n < |u| \ \text{, } |v| \leq 2n$. Then we use the
calculations performed in case 2 of Claim 1 and obtain
\begin{equation}
|f_n(u) -  f_n(v)|_{\Om} \leq |f(u) - f(v)|_{\Om} + \Big(
\int_{\Om} C |v|^{2(p-1)} |u-v|^2 \ dx \Big)^{1/2}.
\end{equation}

Now reiterating the strategy used in Case 1, we obtain the desired
result. As before, the case when $|u| \leq n$ and $ n < |v| \leq 2n$
reduces to case 2.\\

\textbf{Claim 3:} $\displaystyle |f_n(u) -
f(u)|_{L^{\frac{m+1}{m}}(\Om)} \to 0$ as $n \to \infty$ for all $u
\in \HO$. This can be easily seen, since $|f_n(u) - f(u)| =
|f(u)||\eta(u)-1|$ shows that $f_n(u) \to f(u)$ a.e. (because $f$ is
continuous and $\eta \to 0$ as $n \to \infty$. Then we also have
that $|f_n(u)| \leq 2 |f(u)|$ and  $f(u) \in
L^{\frac{m+1}{m}}(\Om)$, for $ u \in \HO$. Thus by Lebesgue
Dominated Convergence Theorem, $f_n \to f$ in
$L^{\frac{m+1}{m}}(\Om)$.

Since $g$ and $f_n$ satisfy the assumptions  of Theorem 2.1, then
its result holds true for each $n$ with $T_M(|(u^n_0,u^n_1)|_{\HO \times
\LO}, l_m)$ (with $T_M$ uniform in $n$), i.e for each $n$, there
exists a pair $(u^n(t),u^n_t(t)) \in C(0,T; \HH)$ which solves the approximated
problem (\ref{1.2}). Thus $u^n(t)$ satisfies the following variational
equality: for any $ \phi \in \HO \cap L^{m+1}(\Om)$, we have:
\begin{equation}\label{variational}
\frac{d}{dt}(u^n_t(t), \phi)_{\Om} + (u^n(t), \phi)_{\HO} +(f_n(u^n(t)), \phi)_{\Om}+
(g(u^n_t(t)), \phi)_{\Om})= 0.
\end{equation}

We will prove that this sequence
of solutions $u^n$ has, on a subsequence, an appropriate limit which is a solution to the original problem (SWB).\\

By using the regularity properties of the solutions $u^n$, we apply
the energy identity to the ``n"-problem and obtain that for each $ 0
< T < T_{max}$, we have
\begin{multline}\label{energyidn}
{1 \over 2} \Big(|u^n_t(T)|^2_{\Om}+ |\nabla u^n(T)|^2_{\Om}\Big) + \int_{\Om_T}
f_n(u^n(t))u^n_t(t) dxdt+
\int_{\Om_T}g(u^n_t(t))u^n_t(t) dxdt 
\\ = {1 \over 2}
\Big(|u^n_t(0)|^2_{\Om}+ |\nabla u^n(0)|^2_{\Om}\Big).
\end{multline}

\textbf{A-priori bounds:} Remember the assumptions on $g$ and $f$:
\begin{itemize}
\item $g(s)s \geq l_m |s|^{m+1} \ \text{for} \ |s| \geq 1$
\item $f_n \ \text{is locally Lipschitz:} \ H^1(\Om) \to L^{m+1 \over
m}(\Om)$
\end{itemize}

Going back to (\ref{energyidn}), we estimate the terms involving the
source $f_n$ by using H\"older's Inequality with $\tilde m =
\frac{m+1}{m}$ and $m+1$, followed by Young's Inequality with the
corresponding components. For simplicity, in the following computations we use $u(t)$ instead of $u^n(t)$.
\begin{equation}\label{estimatefn}
\begin{split}
&\int_{\Om_T} f_n(u(t)) u_t(t) dx \leq \int_0^T |f_n(u(t)|_{\tilde m}
\cdot |u_t(t)|_{m+1}dt \\ &\leq \ep_1 \int_0^T
|u_t(t)|^{m+1}_{m+1}dt
+ C_{\ep_1} \int_0^T |f_n(u)|^{\tilde m}_{\tilde m}dt \\ &\leq \ep_1
\int_0^T |u_t(t)|^{m+1}_{m+1}dt + C_{\ep_1} L^{\tilde m}_{f_n}(K)
\int_0^T |\nabla u(t)|^{\tilde m}_{\Om}dt + C_{\ep_1}C_{f_n} T
\end{split}
\end{equation}

Combining (\ref{energyidn}) with (\ref{estimatefn}) and using the
growth conditions imposed on $g$, we obtain:

\begin{align*}
&{1 \over 2} \Big(|u_t(T)|^2_{\Om}+ |\nabla u(T)|^2_{\Om}\Big) + l_m
\int_0^T |u_t(t)|^{m+1}_{m+1}\ dt - C_{g,f} T \\
& \leq {1 \over 2} \Big(|u_t(0)|^2_{\Om}  +
|\nabla u(0)|^2_{\Om}\Big)+ \ep_1 \int_0^T |u_t(t)|^{m+1}_{m+1} \ dt
+ C_{\ep_1}L^{\tilde m}_f(K)\int_0^T |\nabla u(t)|^{\tilde m}_{\Om} \
dt
\end{align*}

Choosing $ \ep_1 < \frac{l_m}{2}$ and since $m
> 1$, we obtain that for all $T <
T_{max}$, we have
\begin{equation}
|u^n_t(T)|^2_{\Om}+ |\nabla u^n(T)|^2_{\Om} \leq
[|u^n_t(0)|^2_{\Om}+ |\nabla u^n(0)|^2_{\Om} + \overline{C} T] \cdot
e^{C_{l_m,K} T},
\end{equation}

where $\overline{C} = C(g, f, \ep_1, m)$ and $C_{l_m} =
C_{\ep_1}L^{\tilde m}_f(K)$.

Also, we have
\begin{equation}\label{utnorms}
\int_0^T |u^n_t(t)|^{m+1}_{L^{m+1}(\Om)} \leq C_{|u_0|_{\HO},|u_1|_{\Om}, T_{max}}.
\end{equation}

From (\ref{utnorms}), combined with the growth assumptions imposed
on the damping $g$, we obtain that
\begin{align}
& \int_{\Om_T}|g(u^n_t(t))|^{\tilde m} dxdt \leq \int_{\Om_T} L^{\tilde
m}_m|u^n_t(t)|^{m+1} \ dxdt \\ \nonumber & = \int_0^T
|u_t(t)|^{m+1}_{L^{m+1}(\Om)}\ dt \leq C_{|\nabla u_0|_{\Om},
|u_1|_{\Om}, T_{max}}
\end{align}

Therefore, on a subsequence we have
$$(u^n,u^n_t) \to (u^n,u^n_t) \ \text{weakly in}\ \HH$$
$$ u^n_t \to u_t \ \text{weakly in }\ L^{m+1}(0,T; \Om)$$
$$ g(u^n_t) \to g^* \ \text{weakly in}\ L^{\frac{m+1}{m}}(0,T;\Om) \ \text{,
for some} \ g^* \in L^{\frac{m+1}{m}}(0,T;\Om).$$

We want to show that $g^* = g(u_t)$. In order to do that, consider
$u^m$ and $u^n$ be the solutions to the approximated problem
corresponding to the parameters $m$ and $n$. For sake of notation,
let $\tilde u(t) =u^n(t) - u^m(t)$ and $\tilde u_t(t) = u^n_t(t) -
u^m_t(t)$. Then from the energy identity we obtain that for any $T <
T_{max}$ we have {\allowdisplaybreaks
\begin{align}\nonumber &\frac{1}{2}|\tilde u_t(T)|^2_{\Om}+\frac{1}{2} |\tilde u(T)|^2_{\HO} + \int_{\Om_T} (f_n(u^n(t)) - f_m(u^m(t)))\tilde
u_t(t) \ dxdt \\\label{energyidmn} &+\int_{\Om_T} (g(u^n_t(t)) - g(u^m_t(t)))\tilde u_t(t) \ dxdt =0 \end{align}}

First we will show that $ \displaystyle \int_{\Om_T} (f_n(u^n(t)) -
f_m(u^m(t)))(\tilde u_t(t)) \ dxdt \to 0$ as $m, n \to \infty$. Recall
that $\tilde m = \frac{m+1}{m}$ and $|\cdot|_{s} =
|\cdot|_{L^{s}(\Om)}$. Using H\"older's Inequality with $\tilde m$ and
$m+1$, we obtain: {\allowdisplaybreaks
\begin{align*} & \int_{\Om_T}[f_n(u^n(t)) -
f_m(u^m(t))]\tilde u_t(t)dx dt\leq \int_{\Om_T} [f_n(u^n(t))-
f_n(u(t))]\tilde u_t(t)dx dt\\
& + \int_{\Om_T} [f_n(u(t))- f(u(t))]\tilde u_t(t) dxdt + \int_{\Om_T}
[f(u(t))-f_m(u(t))]\tilde u_t(t)dx dt\\
& + \int_{\Om_T} [f_m(u(t))- f_m(u^m(t))]\tilde u_t(t)dx dt \leq
\int_0^T |f_n(u^n(t)) - f_n(u(t))|_{\tilde m} |\tilde u_t(t)|_{m+1}dt\\
& + \int_0^T |f_n(u(t)) - f(u(t))|_{\tilde m} |\tilde
u_t(t)|_{m+1}dt + \int_0^T |f(u(t)) - f_m(u(t))|_{\tilde m} |\tilde u_t(t)|_{m+1}dt\\
& + \int_0^T |f_m(u(t)) - f_m(u_m(t))|_{\tilde m} |\tilde
u_t(t)|_{m+1}dt
\end{align*}}

Now we use the fact that $f$ is locally Lipschitz $H^{1-\ep}(\Om)
\to L^{\tilde m}(\Om)$ and obtain: {\allowdisplaybreaks
\begin{align}
\nonumber & \int_{\Om_T}[f_n(u^n(t)) - f_m(u^m(t))]\tilde u_t(t)dxdt
\leq \int_0^T L(K) |u^n(t) - u(t)|_{H^{1-\ep}(\Om)} |\tilde u_t(t)|_{m+1}dt\\
\nonumber & + \int_0^T |f_n(u(t)) - f(u(t))|_{\tilde m} |\tilde
u_t(t)|_{m+1}dt
 + \int_0^T |f(u(t)) - f_m(u(t))|_{\tilde m} |\tilde u_t(t)|_{m+1}dt\\
\label{abc} & + \int_0^T L(K) |u^m(t) - u(t)|_{H^{1-\ep}(\Om)}
|\tilde u_t(t)|_{m+1}dt.
\end{align}}

We know that $u^n(t) \to u(t)$ weakly in $\HO$ and since the
embedding $H^{1-\ep}(\Om) \subset \HO$ is compact, we get that
$u^n(t) \to u(t)$ strongly in $H^{1-\ep}(\Om)$. We also know that
$\displaystyle |f_n(u) - f(u)|_{\tilde m} \to 0$ as $n \to \infty$
(and same for $m$) and that $|u^n_t(t) - u^m_t(t)|_{m+1, \Om} \leq
C$ for $t < T_{max}$. Thus from (\ref{abc}) we obtain the desired
result
$$ \displaystyle \int_{\Om_T} [f_n(u^n(t)) - f_m(u^m(t))]\tilde u_t(t) \ dxdt \to 0 \ \text{as} \
m, n \to \infty.$$

Now we let $m,n \to \infty$ in (\ref{energyidmn}) and
remembering that $g$ is monotone, we obtain:
\begin{equation}
\lim_{m,n \to \infty} \Big[ |u^n(T) - u^m(T)|^2_{\HO} + |u^n_t(T) -
u^m_t(T)|^2_{\Om} \Big] = 0
\end{equation}
and
\begin{equation}
\lim_{m,n \to \infty} \int_{\Om_T} [g(u^n_t(t)) - g(u^m_t(t))]\tilde
u_t(t) \ dx = 0.
\end{equation}

Since now we know that $u^n_t \to u_t$ weakly in $L^{m+1}(0, T;
\Om)$ and $g(u^n_t) \to g^*$ weakly in $L^{\frac{m+1}{m}}(0,T;\Om)$,
and we also showed that $\limsup_{m,n \to \infty} (g(u^n_t) -
g(u^m_t), u^n_t - u^m_t) \leq 0$, then by Lemma 1.3 (p.42) in
\cite{B}, we obtain that $g^* = g(u_t)$ and $(g(u^n_t),
u^n_t)_{\Om} \to (g(u_t), u_t)_{\Om}$.

Since $\displaystyle |f_n(u) - f(u)|_{\tilde m} \to 0$ as $n \to
\infty$, it follows that $f_n(u^n(t)) \to f(u(t))$ in $L^{\tilde
m}(\Om)$, as $u^n \to u$ weakly in $\HO$.

We are now in the position to pass to the limit in
(\ref{variational}) and obtain the desired equality on bounded domains.

  \subsection{Cutting the initial data}

  Consider now a pair of initial data $ (u_0, u_1)\in H^1(\re^3) \times L^2(\re^3)$ and let $K$ be an upper bound on the energy norm of the initial data, more precisely take $K$ such that
\begin{equation}
|\nabla u_0|_{\re^3}+|u_1|_{\re^3} <K.
\end{equation}

 We find $r$ such that
     \begin{align}\label{whatsr}
     &|\nabla {u}_0|_{B(x_0,r)}<\frac{K}{4},\quad  |u_1|_{B(x_0,r)}<\frac{K}{4},\\
   \notag    &    2(C^* \omega_3)^{\frac{1}{3}}\left(|\nabla
                 {u}_0|_{B(x_0,r)}+|{u}_0|_{B(x_0,r)}\right)
                 < \frac{K}{4},
    \end{align}
    where $  \omega_3$ is the volume of the unit ball in $\re^3$ and $C^*$ is the constant from the Sobolev inequality (which does not depend on $x_0$ nor $r$). It can be easily shown that the above inequalities are satisfied by $ r$ chosen such that
    \begin{align}\label{whatsr1}
   & |{u}_0|_{B(x_0, r)}<\displaystyle \frac{K}{8(C^*  \omega_3)^{\frac{1}{3}}}, \quad |u_1|_{B(x_0,r)}<\frac{K}{4},\\ \nonumber &  |\nabla {u}_0|_{B(x_0, r)}<\min \left\{\displaystyle
       \frac{K}{4}, \displaystyle \frac{K}{8(C^*  \omega_3)^{\frac{1}{3}}}\right\}.
      \end{align}

    The fact that $ r$ can be chosen independently of $x_0$ is motivated by the
    equi-integrability of the functions ${u}_0, \nabla {u}_0, {u}_1$. For each of the functions
${u}_0,\nabla {u}_0, {u}_1$ we apply the following result of classical analysis:\\

    {\it If $f\in L^1(A)$, with $A$ a measurable set, then for every given
$\varepsilon >0$, there exists a number $\delta>0$ such that $\int_{E}| f(x)| dx
<\varepsilon$, for every measurable set $E \subset A$ of measure less than
$\delta$ (see \cite{EG}).} \\

Note that $\delta$ in the above result does not depend on $E$, hence $ r$ does not vary with $x_0$.

From Theorem \ref{t:1} it follows that the solution exists up to time $T(K)$ (which depends on $K$, but it does not depend on $x_0$) on all balls $B(x_0,r), \, x_0\in \re^3$, provided that $u_0\in H_0^1(B(x_0,r))$. In order to obtain that $u_0$ has zero trace on $\partial B(x_0,r)$ we multiply it by a smooth cut-off function $\theta$ such that
\[
\theta(x)=\begin{cases} 1, \quad |x-x_0|\leq  r/2\\
0, \quad |x-x_0|\geq  r
\end{cases}
\]
and
\begin{equation}\label{theta}
 |\theta|_{\infty, B(x_0, r)} \leq 1, \quad
      |\nabla \theta|_{\infty, B(x_0, r)} \leq
\frac{2}{ r}.
\end{equation}
Such $\theta$ can be obtained from a mollification which approximates the Lipschitz function
\[
\theta_0(x)=\begin{cases} 1, \quad |x-x_0|\leq  r/2\\
   2-2 |x-x_0| / r, \quad  r/2\leq |x-x_0|\leq  r\\
0, \quad |x-x_0|\geq  r.
\end{cases}
\]
 We denote by
  \[
  u^{x_0}_0=\theta u_0, \quad u^{x_0}_1=u_1,
  \]
   and by $u^{x_0}$ the solution generated by $(u^{x_0}_0, u^{x_0}_1)$.
  In order to show that
  \begin{equation}\label{small}
|\nabla u_0^{x_0}|_{B(x_0)}+|u_1^{x_0}|_{B(x_0)} <K
\end{equation}
   we start with the following estimate
     \[
      |\nabla u^{x_0}_0|_{2,B(x_0, r)} \leq |\theta|_{\infty, B(x_0, r)}
      |\nabla u_0|_{2,B(x_0, r)} +
      |\nabla \theta|_{\infty,
B(x_0, r)}|u_0|_{2,B(x_0, r)}.
    \]
  By $(\ref{theta}), (\ref{whatsr})$, H\"older's inequality, followed by Sobolev's inequality we conclude that:
    \begin{multline*}
      |\nabla u^{x_0}_0|_{2,B(x_0, r)} < \frac{ K}{4} + |B(x_0, r)|^
      {\frac{1}{3}} \frac{2}{ r}|u_0|_{6,B(x_0, r)}  \\ \leq  \frac{ K}{4} + 2(C^* \omega
_3)^{\frac{1}{3}}\left(|\nabla
      u_0|_{2,B(x_0, r)}+|u_0|_{2,B(x_0, r)}\right)
      \stackrel{(\ref{whatsr})}{\leq} \frac{ K}{4} + \frac{ K}
      {4}= \frac{ K}{2}.
\end{multline*}

Thus we showed that the pair $(u^{x_0}_0, u^{x_0}_1)$ satisfies (\ref{small}).

 \subsection{Patching the small solutions} The key argument that we use in order to construct the solution to the Cauchy problem from the ``partial" solutions to the boundary value problems set on the balls $B(x_0,r)$ constructed in section 3.2 uses an idea due to Crandall and Tartar. They first used this type of argument to obtain global existence of solutions for a Broadwell model with {\it arbitrarily large} initial data starting from solutions with small data (see \cite{T}). Subsequently, the second author has recast it in the framework of semilinear wave equations and showed local existence of solutions for (SW) on the entire space $\re^3$ (see \cite{PR1}); the argument may also be employed on bounded domains as it was done in \cite{PR2}. 
 
 \vspace*{.03in}
\noindent{\bf Step 1. Construction of partial solutions.} Consider a lattice of points in $\re^3$ denoted by $x_j$ situated at distance $d>0$ from each other, such that in every ball of radius $d$ we find at least one $x_j$. Next construct the balls $B_j:=B(x_j,r/2)$, where $r$ is given by (\ref{whatsr}) and inside each $B_j$ take a snapshot of the initial data. More precisely, construct $(u_0^{x_j},u_1^{x_j})$ by the procedure used in subsection 3.2. On each of the balls $B(x_j,r)$ we use theorem \ref{T:existLD1} for the approximated problem given by the system (\ref{1.2}) to obtain existence of solutions $u^{x_j,n}$ up to a time $T(K)$ independent of $x_j$ and of $n$. These solutions will satisfy the estimate (\ref{energyidn}) on $B(x_j,r+T(K))$. Following the arguments from section 3.1 we pass to the limit in the sequence of approximations $u^{x_j,n}$ on each of the balls $B_j$ and obtain a solution $u^{x_j}$.

   \vspace*{.03in}
 \noindent{\bf Step 2. Patching the small solutions.}  For $j\in \mathbb{N}$ let
  \[
   C_j:=\{ (y,s) \in \re^3 \times [0, \infty); |y-x_j| \leq r/2-s\}
 \]
 be the backward cones which have their vertices at $(x_j,r/2)$. For
$d$ small enough (i.e. for $0<d<r/2$) any two neighboring cones $C_j$
and $C_l$ will intersect. 

\vspace*{1.0in}
    \setlength{\unitlength}{.5in}
  \begin{center}
  \hspace{-2.8in}
 \begin{picture}(1,1)
 \put(0,0){\vector(1,0){7}}
 \put(0,1.25){\line(1,0){3.74}}
 \put(0,2){\line(1,0){7}}
 \put(0,0){\vector(0,1){3}}
 \put(3,0){\line(0,1){2}}
 \put(3.75,0){\line(0,1){1.25}}
 \put(4.5,0){\line(0,1){2}}
 \put(3.1,2.2){\vector(1,0){1.45}}
  \put(4,2.5){\makebox(0,0){$d$}}
  \put(1,-0.4){\vector(1,0){2}}
   \put(3.1,2.2){\vector(-1,0){.1}}
    \put(2.1,-0.6){\makebox(0,0){$r/2$}}
  \put(1.1,-.4){\vector(-1,0){.1}}
   \put(-0.2,3){\makebox(0,0){$t$}}
      \put(7,-0.2){\makebox(0,0){$x$}}
    \put(3.5,.25){\makebox(0,0){$I_{j,l}$}}
    \put(1.7,.25){\makebox(0,0){$C_{j}$}}
    \put(5.8,.25){\makebox(0,0){$C_{l}$}}
 \put(3.1,-0.2){\makebox(0,0){$x_j$}}
  \put(4.55,-0.2){\makebox(0,0){$x_l$}}
 \put(-0.7,1.25){\makebox(0,0){$(r-d)/2$}}
   \put(-0.4,2){\makebox(0,0){$r/2$}}
 \thicklines
 \put(1,0){\line(1,1){2}}
  \put(3,2){\line(1,-1){2}}
 \put(2.5,0){\line(1,1){2}}
  \put(4.5,2){\line(1,-1){2}}
\end{picture}\\[.4in]
\begin{center}
  The intersection of the cones $C_j$ and $C_l$\\
   \end{center}
 \end{center}
 
For every set of intersection
 \[
 I_{j,l}: = C_j \cap C_{l}
 \]
 the maximum value for time contained in it is equal to $(r-d)/2$
(see figure above).
 For $t<r/2$ we define the piecewise function:
 \begin{equation}\label{D:soln}
   u(x,t):=u^{x_j}(x,t), \text{ if } (x,t)\in C_j.
 \end{equation}
 This solution is defined only up to time $(r-d)/2$, since the cones
do not cover the entire strip $\re^3 \times (0, r/2)$. By letting
$d\to 0$ we can obtain a solution well defined up to time $r/2$.
 Thus, we have $u$ defined up to time $r/2$, which is the height of
all cones $C_j$. Every pair $(x, t) \in \re^3 \times (0,r/2)$ belongs
to at least one $C_j$, so in order to show that this function from
(\ref{D:soln}) is well defined, we need to check that it is single-valued
on the intersection of two cones. Also, we need to show that the above
function is the solution generated by the pair of initial data $(u_0,
u_1)$. Both proofs will be done in the next step.

 \vspace*{.03in}
\noindent{\bf Step 3. The solution given by (\ref{D:soln}) is valid}. To show that $u$ defined by (\ref{D:soln}) is a proper function, we
use the same result of uniqueness given by the finite speed of propagation.  First note that for $n\geq 3$ the intersection $I_{j,l}$ is not a cone,
but it is contained by the cone $C_{ j,l}$ with the vertex at
$((x_j+x_{l})/2,( r -d)/2)$ of height $( r -d)/2$.   In this cone we
use the uniqueness asserted by the finite speed of propagation as follows.
Recall that the approximations $f_n$ are Lipschitz inside the balls $\{B(x_j,r)\}_{j\in \re^3}$, hence the finite speed property holds for the solutions $u^{j,n}$.
First note that the cones $C_{ j,l}$ contain the sets $I_{ j,l}$, but
$C_{ j,l} \subset C_j \cup C_l$. In $C_j$ and $C_l$ we have the two
solutions $u^{j,n}$ and $u^{l,n}$ hence, in $C_{ j,l}$ we now have defined two
functions which can pose as solutions. Since  $u^{j,n}$ and $u^{l,n}$ start with the same
initial data ($(u_{0}^j, u_{1}^j)=(u_0, u_1)=(u_{0}^l, u_{1}^l)$
on $B_j \cap B_l$), hence they are equal in $C_{ j,l}$; since $I_{ j,l} \subset C_{ j,l}$ we
proved $u^{j,n}=u^{l,n}$ on $I_{j,l}$. By letting $n\to \infty$ we get $u^j=u^l$ on $I_{ j,l}$. Therefore, $u$ is a single-valued (proper) function.

    Finally, the fact that this constructed function $u$ is a solution to
the Cauchy problem (SW) is immediate since it satisfies both, the wave
equation and the initial conditions.
\begin{remark} (\textbf{on global existence}) The above method of using cutoff functions and
    ``patching" solutions based on the finite speed of propagation property
    will work the same way in the case when we have global existence on bounded domains.
    Since we can choose the
height of the the cones as large as we wish the solutions exist globally in time.
\end{remark}

\begin{remark}(\textbf{on uniqueness})
 In \cite{BL_uniq} the authors showed under the same assumptions (A$_f$), (A$_g$) that the boundary value problem (SWB) admits a unique solution (in fact, the result was shown in the presence of damping and source terms in the interior {\it and} on the boundary. The methods employed here seem to preclude us from obtaining a corresponding result for the Cauchy problem (SW) since they are obtained by passing to the limit in sequences of approximations.
\end{remark}
 \end{proof}


       \end{document}